\documentclass[a4paper,11pt]{article}
\usepackage[utf8]{inputenc}
\usepackage[english]{babel}
\usepackage{amsfonts, amssymb, amsthm}
\usepackage{xcolor}
\usepackage{amsmath}
\usepackage{a4wide}
\usepackage[affil-it]{authblk}

\numberwithin{equation}{section}

\usepackage[colorlinks=true,linktocpage,pdfpagelabels,
bookmarksnumbered,bookmarksopen]{hyperref}
\definecolor{ForestGreen}{rgb}{0.1,0.6,0.05}
\definecolor{EgyptBlue}{rgb}{0.063,0.1,0.6}
\hypersetup{
	colorlinks=true,
	linkcolor=EgyptBlue,         
	citecolor=ForestGreen,
	urlcolor=olive
}

\usepackage[hyperpageref]{backref}

\newtheorem{thm}{Theorem}[section]
\newtheorem{lemma}[thm]{Lemma}
\newtheorem{prop}[thm]{Proposition}
\newtheorem{cor}[thm]{Corollary}
\theoremstyle{definition}
\newtheorem{remark}[thm]{Remark}

\usepackage{marginnote}

\title{On full Zakharov equation and its approximations
	\\ \medskip}

\author[1]{Vladimir Bobkov\thanks{E-mail: \texttt{bobkov@kma.zcu.cz}}}
\author[1]{Pavel Dr\'abek\thanks{E-mail: \texttt{pdrabek@kma.zcu.cz}}}
\affil[1]{{\small Department of Mathematics and NTIS, Faculty of Applied Sciences, University of West Bohemia, Univerzitn\'i 8, 306 14 Plze\v{n}, Czech Republic}}
\author[2]{Yavdat Ilyasov\thanks{E-mail: \texttt{ilyasov02@gmail.com}}}
\affil[2]{{\small Institute of Mathematics, Ufa Scientific Center, Russian Academy of Sciences, 112, Chernyshevsky str., 450008 Ufa, Russia}}

\date{}

\begin{document}
\maketitle

\begin{abstract}
	We study the solvability of the Zakharov equation 
	$$
	\Delta^2 u + (\kappa-\omega^2)\Delta u -  \kappa \,\text{div} \left(e^{- |\nabla u|^2} \nabla u\right) = 0
	$$
	in a bounded domain under homogeneous Dirichlet or Navier boundary conditions. This problem is a consequence of the system of equations derived by Zakharov 
	to model the Langmuir collapse in plasma physics.
	Assumptions for the existence and nonexistence of a ground state solution as well as the multiplicity of solutions are discussed. Moreover, we consider formal approximations of the Zakharov equation obtained by the Taylor expansion of the exponential term.
	We illustrate that the existence and nonexistence results are substantially different from the corresponding results for the original problem.

\end{abstract}

\section{Introduction}\label{sec:Introduction}
The main aim of this article is to investigate sets of assumptions which guarantee the existence  of solutions for the  equation
\begin{equation}\label{eq:pro1}
  \Delta^2 u + (\kappa-\omega^2)\Delta u -  \kappa \,\text{div} \left(e^{- |\nabla u|^2} \nabla u\right) = 0
  \quad 
  \text{in } \Omega,
\end{equation}
where $\omega \in \mathbb{R}$, $\kappa > 0$, and $\Omega \subset \mathbb{R}^N$ is a bounded domain.
This equation can be obtained by considering the standing wave solutions of the Zakharov equation \cite[(1.2)]{zakharov1972}
\begin{equation}\label{eq:prototype0}
\Delta\left(i \frac{\partial \psi}{\partial t} + \frac{3}{2} \omega_p r_D^2 \Delta \psi\right) - \frac{\omega_p}{2 n_0} \, \text{div} (\delta n  \, \nabla \psi) = 0,
\end{equation}
introduced for a simplified dynamic description of a plasma turbulence and a development of collapse of appearing low-density regions in the plasma (Langmuir collapse).
Here $\omega_p$ is the frequency of the plasma oscillations, $r_D$ is the Debye radius, $n_0$ is the plasma density (the plasma is assumed to be quasineutral), $\delta n$ is the low-frequency  variation of $n_0$, and $\psi$ is the averaged potential of the high-frequency electric field
$
E = \nabla[\psi \exp(-i\omega_p t) + \overline{\psi} \exp(i\omega_p t)]/2,
$
where $\overline{\psi}$ is a complex conjugate of $\psi$. 

Roughly speaking, under suitable physical assumptions, the equation \eqref{eq:prototype0} can be complemented by connecting $\delta n$ with $\nabla \psi$ via 
the Boltzmann distribution of electrons \cite[(1.4)]{zakharov1972} and the distribution of ions described by the Vlasov equation \cite[(1.5)]{zakharov1972}. 
Zakharov suggested two ways of simplification of the obtained system for different regimes of the plasma: the so-called \textit{static} and \textit{hydrodynamic} approximations. In the present article, we are interested in the static approximation which yields the connection (see \cite[(1.7)]{zakharov1972})
\begin{equation}\label{Boltsman}
\delta n = n_0 \left[\text{exp}\left(-\frac{|\nabla \psi|^2}{16 \pi n_0 (T_i + T_e)}\right)-1 \right],
\end{equation}
where $T_i$ and $T_e$ are the ion and electron temperatures, respectively. 
Substituting \eqref{Boltsman} into \eqref{eq:prototype0}, we deduce the equation
\begin{equation*}
\Delta\left(i \frac{\partial \psi}{\partial t} + \frac{3}{2} \omega_p r_D^2 \Delta \psi\right) - \frac{\omega_p}{2} \, \text{div} \left(\text{exp}\left(-\frac{|\nabla \psi|^2}{16 \pi n_0 (T_i + T_e)}\right) \nabla \psi\right) + \frac{\omega_p}{2} \, \Delta \psi = 0.
\end{equation*}
This equation, under the standing wave solution ansatz $\psi = e^{i \omega^2 t} u$, where $-\omega^2$ is a frequency and $u$ is a real function, leads (in the normalized form) to the stationary Zakharov equation \eqref{eq:pro1}, where $\kappa > 0$ is a fixed parameter defined by the settings of the plasma.

In fact, Zakharov considered the first two terms of the Taylor expansion of the exponential function in \eqref{Boltsman} to get $\delta n \approx - \frac{|\nabla \psi|^2}{16 \pi (T_i + T_e)}$ and, after substitution in \eqref{eq:prototype0}, he obtained the equation \cite[(1.8)]{zakharov1972}
\begin{equation}\label{eq:prototype01}
\Delta\left(i \frac{\partial \psi}{\partial t} + \frac{3}{2} \omega_p r_D^2 \Delta \psi\right) + \frac{\omega_p}{32 \pi n_0 (T_i + T_e)} \, \text{div} (|\nabla \psi|^2 \, \nabla \psi) = 0.
\end{equation}
In particular, after suitable change of variables, we see that the amplitude of the standing wave solution, $u=u(x)$, has to satisfy the normalized equation
\begin{equation}\label{eq:prototype2}
\Delta^2 u -\omega^2 \Delta u + \kappa \, \text{div} \left(|\nabla u|^{2} \, \nabla u\right) 
= 0.
\end{equation}
The following generalization of this equation:
\begin{equation*}\label{eq:prototype3}
\Delta^2 u -\omega^2 \Delta u + \kappa \, \text{div} \left(|\nabla u|^{\sigma} \, \nabla u\right) 
= 0,
\end{equation*}
where $\sigma \in (0, +\infty)$ for $N =2$ and $\sigma \in (0, 4/(N-2))$ for $N \geq 3$, was studied by Colin \cite{colin} and Colin \& Weinstein \cite{colinweinst} in the entire space $\mathbb{R}^N$. They proved that this equation possesses a ground state solution which decays at infinity, and studied its stability.

On the other hand, Dyachenko et al.\ \cite{dyachenko1987,dyachenko1991} worked with the following approximative equation obtained by the first three terms of the Taylor expansion of the exponential function in \eqref{Boltsman}:
\begin{equation}\label{eq:prototype1}
\Delta^2 u -\omega^2 \Delta u + \kappa \, \text{div} \left(|\nabla u|^2 \nabla u\right) 
-\frac{\kappa}{2} \, \text{div} \left(|\nabla u|^4 \nabla u\right)
= 0.
\end{equation}
The last term on the left-hand side of \eqref{eq:prototype1} corresponds to higher-order nonlinear effects which can significantly affect the process of the Langmuir collapse.

Notice that numerical simulations indicate that solutions which correspond to the Langmuir collapse are not radially symmetric, and, in general, have no spatial symmetry, see, e.g., \cite{Degtyarev, dyachenko1987, dyachenko1991, Robinson1989, Robinson}.
From this point of view, it makes sense to consider the Zakharov equation in a bounded domain, which can be seen as an approximation of the entire space case.
Let us remark that the hydrodynamic approximation of the Zakharov system of equations in a bounded domain was treated in \cite{CS}.

\smallskip
In the present article, we look for solutions of the equation \eqref{eq:pro1} which, unlike \eqref{eq:prototype2} and \eqref{eq:prototype1}, includes the full Boltzmann distribution \eqref{Boltsman}. 
We consider \eqref{eq:pro1} in a bounded domain $\Omega \subset \mathbb{R}^N$, $N \geq 1$, whose boundary $\partial \Omega$ is of class $C^{4,\eta}$, $\eta \in (0,1)$ (see, e.g., \cite[p.\ 94]{giltrud}).
We will find solutions of \eqref{eq:pro1} which satisfy the following two basic types of homogeneous boundary conditions:
\begin{alignat}{2}
\label{eq:Dirichlet}
&\text{Dirichlet boundary condition:}&
\quad
&u = \frac{\partial u}{\partial \nu} = 0 \quad \text{on } \partial \Omega, \\
\label{eq:Navier}
&\text{Navier boundary condition:}&
\quad
&u =  \Delta u	= 0 \quad \text{on } \partial \Omega.
\end{alignat}
Here $\nu$ is the outward unit normal vector to $\partial \Omega$.
For a working space for the problem \eqref{eq:pro1}, \eqref{eq:Dirichlet} we set $X := W^{2,2}_0(\Omega)$, while for the problem \eqref{eq:pro1}, \eqref{eq:Navier} we set $X := W^{2,2}(\Omega) \cap W_0^{1,2}(\Omega)$ under an additional assumption that $\partial \Omega$ satisfies a uniform outer ball condition (see, e.g., \cite[p.~52]{gazzola2010}). 
Hereinafter, by a solution of \eqref{eq:pro1} we will understand a solution of the boundary value problem \eqref{eq:pro1}, \eqref{eq:Dirichlet} or \eqref{eq:pro1}, \eqref{eq:Navier}, omitting, for simplicity, the reference to the boundary conditions.
More precisely, we say that $u \in X$ is a (weak) solution of \eqref{eq:pro1} if $u$ is a critical point of the energy functional $E_\omega \in C^2(X, \mathbb{R})$ defined by
$$
E_\omega(u) 
= 
\frac{1}{2} \int_\Omega |\Delta u|^2 \, dx
-\frac{\kappa-\omega^2}{2} \int_\Omega |\nabla u|^2 \, dx
+ \frac{\kappa}{2} \int_\Omega \left(1 - e^{- |\nabla u|^2}\right) dx,
$$
that is, 
\begin{equation}\label{eq:sol}
\left<E'_\omega(u),\varphi\right> = 
\int_\Omega \Delta u \Delta \varphi \, dx
- (\kappa-\omega^2) \int_\Omega (\nabla u, \nabla \varphi) \, dx
+ \kappa \int_\Omega e^{- |\nabla u|^2} (\nabla u, \nabla \varphi) \, dx = 0
\end{equation}
for all $\varphi \in X$. 
Here $\left<\cdot,\cdot\right>$ stands for the dual pairing between $X$ and its dual space $X^*$, and $(\cdot,\cdot)$ denotes the scalar product in $\mathbb{R}^N$.
Note that if a weak solution $u$ belongs to $C^{4}(\overline{\Omega})$, then $u$ is a classical solution, that is, the equation \eqref{eq:pro1} and boundary conditions are satisfied pointwise in $\overline{\Omega}$.
We will say that a nonzero weak solution $u$ is a \textit{ground state solution} whenever
$$
E_\omega(u) \leq E_\omega(v) 
\quad \text{for any nonzero weak solution } v \text{ of } \eqref{eq:pro1}.
$$

In the present article we prove the following result on the existence and nonexistence of solutions to \eqref{eq:pro1}. Denote by $\lambda_k$ the sequence of eigenvalues of the problem \eqref{eq:eigenvalue}, see Section~\ref{sec:preliminaries} below.

\begin{thm}\label{thm:1}
	Assume that $\kappa-\lambda_{k+1} < \omega^2 < \kappa-\lambda_k$ for some $k \in \mathbb{N}$.
	Then \eqref{eq:pro1} possesses a ground state solution $u \in C^{4,\gamma}(\overline{\Omega})$ with $E_\omega(u) > 0$, which is the critical point of $E_\omega$ of saddle type.
	Moreover, if  $\omega^2 \geq \kappa - \lambda_1$, then \eqref{eq:pro1} has no nonzero solutions.
\end{thm}

In Propositions \ref{prop:approx_first} and \ref{prop:approx_two} below, we also prove the corresponding existence results for the approximative problems \eqref{eq:prototype2} and \eqref{eq:prototype1}. These results show that the set of solutions for \eqref{eq:pro1} can be significantly different than the sets of solutions for \eqref{eq:prototype2} and \eqref{eq:prototype1}.

In addition, we provide the following multiplicity result for the problem \eqref{eq:pro1}.
\begin{thm}\label{thm:2}
	Assume that $\kappa-\lambda_{k+1} < \omega^2 < \kappa-\lambda_k$ for some $k \in \mathbb{N}$. Then \eqref{eq:pro1} has at least $k$ nonzero solutions with positive energy which enjoy $C^{4,\gamma}(\overline{\Omega})$-regularity. These solutions are critical points of $E_\omega$ of saddle type.
\end{thm}

\smallskip
The article is organized as follows. In Section \ref{sec:preliminaries}, we give some preliminary information. We also show the nonexistence part of Theorem \ref{thm:1}, see Remark \ref{rem:nonexistence}. 
Section \ref{sec:ground_state} contains the existence part of Theorem \ref{thm:1}. 
The application of the direct minimization over the Nehari manifold lets to significant problems which we overcome by using the mountain-pass theorem due to Ambosetti and Rabinowitz \cite{AR}.
In Section \ref{sec:bound_state}, we prove the multiplicity result of Theorem \ref{thm:2} by means of an abstract critical point theorem due to Bartolo, Benci and Fortunato \cite{BBF}. 
Section \ref{sec:regularity} is devoted to the regularity of obtained solutions. Here we use a bootstrap argument combined with suitable embeddings of Sobolev spaces.
Finally, in Section \ref{sec:comparison}, we show the difference between \eqref{eq:pro1} and its formal approximations \eqref{eq:prototype2} and \eqref{eq:prototype1}. In contrast with the original problem \eqref{eq:pro1}, to analyze \eqref{eq:prototype2} we are able to use the Nehari minimization approach and obtain the existence result in Proposition \ref{prop:approx_first}. The equation \eqref{eq:prototype1} is treated via the global minimization method, see Proposition \ref{prop:approx_two}.

\section{Preliminaries}\label{sec:preliminaries}

Hereinafter, we denote by $L^p(\Omega)$ the standard Lebesgue space endowed with the norm $\|u\|_{L^p}:=\left(\int_\Omega |u|^p \, dx\right)^{1/p}$, $p>1$.
Let $\alpha = (\alpha_1, \dots, \alpha_N)$ be a multi-index of length $|\alpha| = \sum_{j=1}^{N} \alpha_j$. Denote the differential operator of order $|\alpha|$ as $D^\alpha = \partial^{\alpha_1}/\partial x_1^{\alpha_1} \cdots \partial^{\alpha_N}/\partial x_N^{\alpha_N}$.
We will work with the Sobolev spaces
$W^{k,p}(\Omega) = \{u: D^\alpha u \in L^p(\Omega) \text{ for all } |\alpha| \leq k\}$.
Note that $W^{k,p}(\Omega)$ is a reflexive Banach space which can be endowed (using the Gagliardo--Nirenberg interpolation inequality) with the norm
\begin{equation}\label{eq:norm}
\|u\|_{W^{k,p}} = \|u\|_{L^p} + \sum_{|\alpha| = k}\|D^\alpha u\|_{L^p}, 
\end{equation}
i.e., the intermediate derivatives are not involved. 
We denote by $W_0^{k,p}(\Omega)$ the closure of $C_0^\infty(\Omega)$ with respect to the norm \eqref{eq:norm}.
Notice, in particular, that $W_0^{2,2}(\Omega)$ is the Hilbert space which can be endowed with the equivalent norm 
$\|\Delta u\|_{L^2}$, see \cite[Theorem 2.2, p.\ 32]{gazzola2010}. Moreover, since we assume that $\partial \Omega \in C^{4,\eta}$ and satisfies a uniform outer ball condition, the Hilbert space $W^{2,2}(\Omega) \cap W_0^{1,2}(\Omega)$ can be also endowed with the norm $\|\Delta u\|_{L^2}$, see \cite[Theorem 2.31, p.\ 52]{gazzola2010}. We will mainly work with this norm.

We often employ the following compact embeddings (see, e.g., \cite[p.\ 168]{AF}):
\begin{itemize}
	\item If $N = 1, 2$, then $W^{2,2}(\Omega) \hookrightarrow\hookrightarrow W^{1,q}(\Omega)$ for any $q \in (1, +\infty)$.
	\item If $N \geq 3$, then $W^{2,2}(\Omega) \hookrightarrow\hookrightarrow W^{1,q}(\Omega)$ for any $q \in (1, 2^*)$, where $2^* := \frac{2N}{N-2}$.
\end{itemize}

\smallskip
Consider the eigenvalue problem
\begin{equation}\label{eq:eigenvalue}
\left\{
\begin{aligned}
&\Delta^2 u + \lambda \Delta u = 0 \quad \text{in } \Omega,\\
&u \text{ satisfies either } \eqref{eq:Dirichlet} \text{ or } \eqref{eq:Navier}.
\end{aligned}
\right.
\end{equation}
By the standard Courant-Fisher variational principle we can find a sequence of eigenvalues $\{\lambda_k\}_{k \in \mathbb{N}}$ and the sequence of associated eigenfunctions $\{\varphi_k\}_{k \in \mathbb{N}}$ such that $\{\varphi_1, \varphi_2, \dots\}$ forms an orthonormal basis of $X$.
In particular, the first eigenvalue is given by 
\begin{equation}\label{eq:mu1}
\lambda_1 := \inf_{v \in X \setminus \{0\}} \frac{\int_\Omega |\Delta v|^2 \, dx}{\int_\Omega |\nabla v|^2 \, dx}.
\end{equation}
Note that $\lambda_1 > 0$, as it follows from the above discussion about the equivalence of norms. 

\smallskip
It is easy to see from \eqref{eq:sol} that any nonzero critical point of $E_\omega$ belongs to the Nehari manifold 
\begin{equation*}\label{eq:nehari}
\mathcal{N}_\omega 
= 
\left\{
u \in  X\setminus \{0\}:  
\int_\Omega |\Delta u|^2 \, dx
-(\kappa-\omega^2) \int_\Omega |\nabla u|^2 \, dx 
+ \kappa \int_\Omega e^{-|\nabla u|^2} |\nabla u|^2 \, dx=0
\right\}.
\end{equation*}

\begin{lemma}\label{lem:ExNeh}
	Let $u \in X\setminus \{0\}$. There exists $t_u > 0$ such that $t_u u \in \mathcal{N}_\omega$ if and only if
	\begin{equation}\label{eq:H<0}
	\int_\Omega |\Delta u|^2 \, dx
	-(\kappa-\omega^2) \int_\Omega |\nabla u|^2 \, dx < 0.
	\end{equation}	
	Moreover, if $u$ satisfies \eqref{eq:H<0}, then there exists $\bar{t}_u > 0$ such that $E_\omega(t u) < 0$ for all $t > \bar{t}_u$.
\end{lemma}
\begin{proof}
	Consider the fibering functional
	\begin{equation}\label{eq:fibering}
	E_\omega(tu) 
	= 
	\frac{t^2}{2} \left(\int_\Omega |\Delta u|^2 \, dx
	-(\kappa-\omega^2) \int_\Omega |\nabla u|^2 \, dx \right)
	+ \frac{\kappa}{2} \int_\Omega \left(1 - e^{- t^2 |\nabla u|^2}\right) dx,
	\end{equation}
	where $t \geq 0$.
	We see that 
	$$
	\frac{\partial}{\partial t}E_\omega(tu)
	= 
	t \left( \int_\Omega |\Delta u|^2 \, dx
	-(\kappa-\omega^2) \int_\Omega |\nabla u|^2 \, dx \right)
	+ \kappa t \int_\Omega e^{- t^2 |\nabla u|^2} |\nabla u|^2 \, dx,
	$$ 
	and hence $\mathcal{N}_\omega$ can be equivalently defined as 
	$$
	\mathcal{N}_\omega
	= 
	\left\{
	u \in  X \setminus \{0\}:~ \frac{\partial}{\partial t}E_\omega(tu) = 0 \text{ at } t = 1
	\right\}.
	$$
	
	Assume first that $u \in X\setminus \{0\}$ satisfies \eqref{eq:H<0}. 
	It is not hard to see that
	\begin{align*}
	&\frac{1}{t}\,\frac{\partial}{\partial t}E_\omega(t u) \to \int_\Omega |\Delta u|^2 \, dx +
	\omega^2 \int_\Omega |\nabla u|^2 \, dx > 0 
	&&\text{as } t \to 0,\\ 
	&\frac{1}{t}\,\frac{\partial}{\partial t}E_\omega(t u) \to \int_\Omega |\Delta u|^2 \, dx
	-(\kappa-\omega^2) \int_\Omega |\nabla u|^2 \, dx < 0
	&&\text{as } t \to +\infty.
	\end{align*}
	Thus, by continuity, there exists $t_u > 0$ such that $\frac{\partial}{\partial t}E_\omega(tu) = 0$ at $t=t_u$, which implies that $t_u u \in \mathcal{N}_\omega$.
	
	On the other hand, if \eqref{eq:H<0} does not hold, then $\frac{\partial}{\partial t}E_\omega(t u) > 0$ for all $t>0$, that is, $tu$ cannot belong to $\mathcal{N}_\omega$.
	
	If we assume that some $u$ satisfies \eqref{eq:H<0}, we can use the boundedness of the last integral in \eqref{eq:fibering} to deduce the existence of $\bar{t}_u > 0$ such that $E_\omega(t u) < 0$ for all $t > \bar{t}_u$.
\end{proof}

\begin{remark}\label{rem:nonexistence}
	The definition \eqref{eq:mu1} of the first eigenvalue $\lambda_1$ implies that \eqref{eq:H<0} can be satisfied if and only if $\kappa - \omega^2 > \lambda_1$. Thus,  $\mathcal{N}_\omega \neq \emptyset$ if and only if $\kappa - \omega^2 > \lambda_1$. 
	Therefore, recalling that any critical point of $E_\omega$ belongs to $\mathcal{N}_\omega$, we derive the nonexistence part of Theorem \ref{thm:1}.
\end{remark}

\begin{lemma}\label{lem:level}
	Any nonzero critical point $u$ of $E_\omega$ satisfies $E_\omega(u) \in \left(0, \frac{\kappa|\Omega|}{2}\right)$.
\end{lemma}
\begin{proof}
	Let $u$ be a nonzero critical point of $E_\omega$. 
	Recall that $u \in \mathcal{N}_\omega$ and notice that $0 < 1 - e^{- s^2}\left(1 + s^2\right) < 1$ for all $s>0$. Then, we derive the desired fact:
	$$
	E_\omega(u) = 
	E_\omega(u) - \frac{1}{2}\left<E'_\omega(u),u\right> = 
	\frac{\kappa}{2} \int_\Omega \left(1 - e^{- |\nabla u|^2}\left(1 + |\nabla u|^2\right) \right) dx \in \left(0, \frac{\kappa|\Omega|}{2}\right).
	$$
\end{proof}

\section{Ground state solution}\label{sec:ground_state}

In this section, we obtain the existence result of Theorem \ref{thm:1}. 
We begin with an auxiliary fact about the Palais--Smale condition for $E_\omega$.

\begin{lemma}\label{lem:PS}
	Assume that $(\kappa-\omega^2)$ is not an eigenvalue of the problem \eqref{eq:eigenvalue}.
	Then $E_\omega$ satisfies the following form of the Palais--Smale condition: if $\{u_n\}_{n \in \mathbb{N}} \subset X$ is such that 
	\begin{equation}\label{eq:PS}
	\|E_\omega'(u_n)\|_{X^*} \to 0 
	\quad \text{as }
	n \to +\infty,
	\end{equation}
	then $\{u_n\}_{n \in \mathbb{N}}$ has a strongly convergent subsequence in $X$ to a critical point of $E_\omega$.
\end{lemma}
\begin{proof}
	Let $\{u_n\}_{n \in \mathbb{N}}$ satisfy \eqref{eq:PS}.
	First we prove that $\{u_n\}_{n \in \mathbb{N}}$ is bounded in $X$. Suppose, by contradiction, that $\|\Delta u_n\|_{L^2} \to +\infty$ as $n \to +\infty$. 
	Let $\{v_n\}_{n \in \mathbb{N}}$ be a normalized sequence such that $u_n = \|\Delta u_n\|_{L^2} \,  v_n$ for $n \in \mathbb{N}$. 
	Then, we see from \eqref{eq:PS} that
	\begin{align}
	\label{eq:normalized_energy_tend_to_zero}
	&\frac{\left|\left<E_\omega'(u_n), \varphi\right>\right|}{\|\Delta u_n\|_{L^2}} \\
	\notag
	&= \left|\int_\Omega \Delta v_n \Delta \varphi \, dx
	-(\kappa-\omega^2) \int_\Omega (\nabla v_n, \nabla \varphi)\,dx
	+\kappa \int_\Omega e^{- \|\Delta u_n\|_{L^2}^2 |\nabla v_n|^2} (\nabla v_n, \nabla \varphi) \, dx\right|
	\to 0
	\end{align}
	as $n \to +\infty$ for any $\varphi \in X$. 
	Let us show that the last integral converges to zero. 		
	To this end, we take any $C > 0$ and make a decomposition
	\begin{align*}
	&\int_\Omega e^{- \|\Delta u_n\|_{L^2}^2 |\nabla v_n|^2} (\nabla v_n, \nabla \varphi) \, dx \\
	&= \int_{\{|\nabla v_n| \leq C\}} e^{- \|\Delta u_n\|_{L^2}^2 |\nabla v_n|^2} (\nabla v_n, \nabla \varphi) \,dx 
	+ 
	\int_{\{|\nabla v_n| > C\}} e^{- \|\Delta u_n\|_{L^2}^2 |\nabla v_n|^2} (\nabla v_n, \nabla \varphi) \, dx.
	\end{align*}
	Since $e^{-s^2} \leq 1$ for all $s \in \mathbb{R}$, we obtain 
	\begin{align*}
	\bigg|\int_{\{|\nabla v_n| \leq C\}} e^{- \|\Delta u_n\|_{L^2}^2 |\nabla v_n|^2} (\nabla v_n, \nabla \varphi) \, dx \bigg| \leq
	\int_{\{|\nabla v_n| \leq C\}} |\nabla v_n| |\nabla \varphi| \, dx  
	\\
	\leq 
	C \int_{\{|\nabla v_n| \leq C\}} |\nabla \varphi| \, dx  \leq C \, |\Omega|^{1/2} \, \|\nabla \varphi\|_{L^2}.	
	\end{align*}
	On the other hand, since $e^{-s^2}$ is strictly decreasing for $s>0$, we get
	\begin{align*}
	\bigg|\int_{\{|\nabla v_n| > C\}} e^{- \|\Delta u_n\|_{L^2}^2 |\nabla v_n|^2} (\nabla v_n, \nabla \varphi) \, dx\bigg|
	\leq 
	e^{- \|\Delta u_n\|_{L^2}^2 C^2} \int_{\{|\nabla v_n| > C\}} |\nabla v_n| |\nabla \varphi| \, dx \\
	\leq e^{- \|\Delta u_n\|_{L^2}^2 C^2} \|\nabla v_n\|_{L^2} \|\nabla \varphi\|_{L^2} \leq c_1 \,  e^{- \|\Delta u_n\|_{L^2}^2 C^2},
	\end{align*}
	where $c_1 > 0$ depends neither on $C$ nor on $n \in \mathbb{N}$ due to $\|\Delta v_n\|_{L^2} = 1$.
	Therefore, taking $C = \|\Delta u_n\|_{L^2}^{-1/2}$ and letting $n \to +\infty$, we get
	\begin{align*}
	\left|\int_\Omega e^{- \|\Delta u_n\|_{L^2}^2 |\nabla v_n|^2}  (\nabla v_n, \nabla \varphi) \, dx \right| 
	&\leq \|\Delta u_n\|_{L^2}^{-1/2}\, |\Omega|^{1/2} \,\|\nabla \varphi\|_{L^2} + c_1 \, e^{- \|\Delta u_n\|_{L^2}}
	\to 0.
	\end{align*}
	
	Thus, \eqref{eq:normalized_energy_tend_to_zero} leads to 
	\begin{equation}\label{eq:v_eigen}
	\int_\Omega \Delta v_n \Delta \varphi \, dx 
	-(\kappa-\omega^2) \int_\Omega (\nabla v_n, \nabla \varphi) \, dx \to 0
	\quad \text{for all } \varphi \in X
	\text{ as } n \to +\infty.
	\end{equation}
	Since $\{v_n\}_{n \in \mathbb{N}}$ is bounded in $X$, 
	the Banach--Alaoglu and Sobolev theorems imply that $v_n$ converges (up to a subsequence) weakly in $X$ and strongly in $W^{1,2}(\Omega)$ to some $v \in X$.
	Note that $v \not\equiv 0$ a.e.\ in $\Omega$. Indeed, we deduce from \eqref{eq:PS} that
	$$
	\frac{\left|\left<E_\omega'(u_n), v_n\right>\right|}{\|\Delta u_n\|_{L^2}} 
	= 
	\left|
	1
	-(\kappa-\omega^2) \int_\Omega |\nabla v_n|^2 \, dx
	+\kappa \int_\Omega e^{- \|\Delta u_n\|_{L^2}^2 |\nabla v_n|^2} |\nabla v_n|^2 \, dx
	\right|
	\to 0
	$$
	as $n \to +\infty$. Arguing as above, we see that the last integral converges to zero. This implies that $\|\nabla v_n\|_{L^2} > c_2 > 0$ for some $c_2 > 0$ and all $n \in \mathbb{N}$, which gives $v \not\equiv 0$ a.e.\ in $\Omega$ due to the strong convergence in $W^{1,2}(\Omega)$.
	Therefore, we conclude from \eqref{eq:v_eigen} that $(\kappa-\omega^2)$ is an eigenvalue of the problem \eqref{eq:eigenvalue} with the associated eigenfunction $v$. However, this contradicts the assumption of the lemma and hence it proves that $\{u_n\}_{n \in \mathbb{N}}$ is bounded in $X$.
	
	Since $\{u_n\}_{n \in \mathbb{N}}$ is bounded, we can extract a subsequence which converges weakly in $X$ and strongly in $W^{1,2}(\Omega)$ to some function $u \in X$. 	
	If we assume that $\|\Delta u_n\|_{L^2} \to 0$ as $n \to +\infty$, then $u_n$ converges strongly in $X$ to a zero critical point of $E_\omega$ and the proof is complete.
	Now assume that $\|\Delta u_n\|_{L^2} > c_3 >0$ for some $c_3 > 0$ and all $n \in \mathbb{N}$.
	Let us prove that $u \not\equiv 0$ a.e.\ in $\Omega$.  
	By compactness of the embedding $X \hookrightarrow\hookrightarrow W^{1,2}(\Omega)$ it is enough to show that there is $c_4 > 0$ such that for all $n \in \mathbb{N}$ we have $\|\nabla u_n\|_{L^2} \geq c_4$.
	Suppose, by contradiction, that $\|\nabla u_n\|_{L^2} \to 0$ as $n \to +\infty$.
	Recalling that $\{u_n\}_{n \in \mathbb{N}}$ is bounded, we deduce from \eqref{eq:PS} that 
	$$
	\left|\int_\Omega |\Delta u_n|^2 \, dx -
	(\kappa-\omega^2)\int_\Omega |\nabla u_n|^2 \, dx
	+\kappa \int_\Omega e^{- |\nabla u_n|^2} |\nabla u_n|^2 \, dx\right|
	=
	\left|\left<E_\omega'(u_n), u_n\right>\right|
	\to 0
	$$
	as $n \to +\infty$, which contradicts $\|\Delta u_n\|_{L^2} > c_3 >0$. 
	Therefore, $u \not\equiv 0$ a.e.\ in $\Omega$.
	
	Finally, let us show that $u_n \to u$ strongly in $X$.
	First, we see from \eqref{eq:PS} and the boundedness of $\{u_n\}_{n \in \mathbb{N}}$ in $X$ that 
	$$
	\left|\left< E_\omega'(u_n), u_n-u \right>\right| \to 0
	\quad \text{as } n \to +\infty.
	$$
	Therefore, in view of the strong convergence of $u_n$ to $u$ in $W^{1,2}(\Omega)$, we get
	\begin{equation}\label{eq:conv_1}
	\int_\Omega \Delta u_n \, \Delta(u_n-u) \, dx \to 0
	\quad \text{as } n \to +\infty.
	\end{equation}
	On the other hand, the weak convergence of $u_n$ to $u$ in $X$ implies 
	\begin{equation}\label{eq:conv_2}
	\int_\Omega \Delta u \, \Delta(u_n-u) \, dx \to 0
	\quad \text{as } n \to +\infty.
	\end{equation}
	Thus, subtracting \eqref{eq:conv_2} from \eqref{eq:conv_1}, we get the desired strong convergence of $u_n$ to $u$ in $X$.	
\end{proof}

The existence part of Theorem \ref{thm:1} is given in the following proposition.
\begin{prop}\label{lem:1}
	Assume that $\lambda_{k} < \kappa - \omega^2 < \lambda_{k+1}$ for some $k \in \mathbb{N}$.
	Then there exists a nonzero ground state solution of \eqref{eq:pro1} with positive energy, which is critical point of $E_\omega$ of saddle type.
\end{prop}
\begin{proof}
	To obtain the claim, we will show that $E_\omega$ has the mountain-pass geometry. 
	Since $E_\omega(0)=0$, it is sufficient to check assumptions $(A_1)$ and $(A_2)$ below.
	
	$(A_1)$ 
	\textit{There exists $\rho>0$ and $\alpha>0$ such that $E_\omega(u) \geq \alpha$ for all $u \in X$ with $\|\Delta u\|_{L^2} = \rho$.}
	
	We argue by contradiction. Let $\rho>0$ be arbitrary and for any $n \in \mathbb{N}$ there exists $u_n \in X$ such that $\|\Delta u_n\|_{L^2} = \rho$ and $E_\omega(u_n) < 1/n$. 
	By the Banach--Alaoglu and Sobolev theorems we may assume that $u_n$ converges (up to a subsequence) to some $u_\rho \in X$ weakly in $X$ and strongly in $W^{1,2}(\Omega)$. Thus, $\|\Delta u_\rho\|_{L^2} \leq \rho$ and $E_\omega(u_\rho) \leq 0$ by the weak lower-semicontinuity of $E_\omega$. Moreover, $u_\rho$ is not zero. Indeed, 
	$$
	\frac{\rho^2}{2} - \frac{\kappa - \omega^2}{2} \|\nabla u_n\|_{L^2}^2 < E_\omega(u_n) < \frac{1}{n},
	$$
	which yields $\|\nabla u_n\|_{L^2} \geq c > 0$ for some $c > 0$ and all $n \in \mathbb{N}$. Consequently, the strong convergence in $W^{1,2}(\Omega)$ implies that $u_\rho \not\equiv 0$ a.e.\ in $\Omega$. 
	Moreover, the inequality $E_\omega(u_\rho) \leq 0$ leads to
	$$
	\int_\Omega |\Delta u_\rho|^2 \, dx 
	-(\kappa-\omega^2) \int_\Omega |\nabla u_\rho|^2 \, dx \leq 
	- \kappa \int_\Omega \left(1 - e^{- |\nabla u_\rho|^2}\right) dx < 0.
	$$
	Hence, applying Lemma \ref{lem:ExNeh}, we can find $t_\rho \in (0,1)$ such that $t_\rho u_\rho \in \mathcal{N}_\omega$. Note that $\|\Delta (t_\rho u_\rho)\|_{L^2} \leq \rho$.
	
	Choose now $\rho = 1/m$ and set $w_m := t_{1/m} u_{1/m}$, $m \in \mathbb{N}$.
	Then $w_m \in \mathcal{N}_\omega$ and $\|\Delta w_m\|_{L^2} \leq 1/m$.
	We normalize $w_m$ as $w_m = \|\Delta w_m\|_{L^2} \, v_m$ where  $\|\Delta v_m\|_{L^2} = 1$ for each $m \in \mathbb{N}$. 
	Let us show that
	\begin{equation*}
	F(m) := \int_\Omega |\nabla v_m|^2 \left(e^{- \|\Delta w_m\|_{L^2}^2 |\nabla v_m|^2} - 1\right) dx
	\to 0
	\quad \text{as }
	m \to +\infty.
	\end{equation*}
	To this end, we take any $C > 0$ and split the integral over $\Omega$ as follows:
	\begin{align*}
	F(m) &= \int_{\{|\nabla v_m| \leq C\}} |\nabla v_m|^2 \left(e^{- \|\Delta w_m\|_{L^2}^2 |\nabla v_m|^2} - 1\right) dx \\
	&+ 
	\int_{\{|\nabla v_m| > C\}} |\nabla v_m|^2 \left(e^{- \|\Delta w_m\|_{L^2}^2 |\nabla v_m|^2} - 1\right) dx.
	\end{align*}
	Since $e^{-s^2} \leq 1$ for $s \in \mathbb{R}$, $e^{-s^2}$ is strictly decreasing for $s \geq 0$, and due to $\|\Delta w_m\|_{L^2} \to 0$, we obtain 
	\begin{equation}
	\label{eq:est1}
	\bigg|\int_{\{|\nabla v_m| \leq C\}} |\nabla v_m|^2 \left(e^{- \|\Delta w_m\|_{L^2}^2 |\nabla v_m|^2} - 1\right) dx \bigg| 
	\leq 
	\left(1 - e^{- \|\Delta w_m\|_{L^2}^2 C^2} \right) \int_{\Omega} |\nabla v_m|^2 \, dx \to 0
	\end{equation}
	as $m \to +\infty$.
	On the other hand, we get
	\begin{equation}
	\label{eq:est2}
	\bigg|\int_{\{|\nabla v_m| > C\}} |\nabla v_m|^2 \left(e^{- \|\Delta w_m\|_{L^2}^2 |\nabla v_m|^2} - 1\right) dx\bigg|
	\leq 
	2 \int_{\{|\nabla v_m| > C\}} |\nabla v_m|^2 \, dx
	\end{equation}
	for all $m \in \mathbb{N}$, and
	\begin{equation}
	\label{eq:est3}
	\int_{\{|\nabla v_m| > C\}} |\nabla v_m|^2 \, dx 
	\to
	\int_{\{|\nabla v| > C\}} |\nabla v|^2 \, dx 
	\quad
	\text{as }	
	m \to +\infty.
	\end{equation}
	Indeed, to justify the last convergence, let us denote by $j_m(x) := \chi_{\{|\nabla v_m(x)| > C\}}$ the characteristic function of the set $\{|\nabla v_m| > C\}$. 
	We show that $|\nabla v_m|\,j_m \to |\nabla v|\,j$ strongly in $L^2(\Omega)$.
	We have
	\begin{align*}
	\int_\Omega \left(|\nabla v_m|\,j_m - |\nabla v|\,j\right)^2 dx 
	&= 
	\int_\Omega \left(\left(|\nabla v_m| - |\nabla v|\right) j_m + |\nabla v| \left(j_m - j \right)\right)^2 dx \\
	&\leq
	2\int_\Omega \left(|\nabla v_m| - |\nabla v|\right)^2 \, dx + 2\int_\Omega |\nabla v|^2 \left(j_m - j\right)^2 dx.
	\end{align*}
	The first integral on the right-hand side converges to $0$ due to the strong convergence $|\nabla v_m| \to |\nabla v|$ in $L^2(\Omega)$. 
	At the same time, we have $|\nabla v_m| \to |\nabla v|$ a.e.\ in  $\Omega$, and hence $j_m \to j$ and $|\nabla v|^2 \left(j_m - j\right) \to 0$ a.e.\ in $\Omega$. 
	Therefore, applying the Lebesgue dominated convergence theorem, we deduce that the second integral on the right-hand side also converges to $0$. 
	Thus, we have obtained the desired claim.
	
	Finally, \eqref{eq:est1}, \eqref{eq:est2}, and \eqref{eq:est3} yield
	\begin{equation*}
	\lim\limits_{m \to +\infty}|F(m)| 
	\leq 2 \int_{\{|\nabla v| > C\}} |\nabla v|^2 \, dx
	\quad
	\text{for any }
	C > 0.
	\end{equation*}	
	Recalling that $|\nabla v| \in L^2(\Omega)$, we let $C \to +\infty$ to conclude that $\lim\limits_{m \to +\infty}|F(m)| = 0$.	
	Therefore, since $w_m \in \mathcal{N}_\omega$ for each $m \in \mathbb{N}$, we get
	\begin{align*}
	0 
	= \frac{1}{\|\Delta w_m\|_{L^2}} \left.\frac{\partial}{\partial t} E_\omega(t v_m)\right|_{t=\|\Delta w_m\|_{L^2}}
	&=
	\int_\Omega |\Delta v_m|^2 \, dx
	-(\kappa-\omega^2) \int_\Omega |\nabla v_m|^2 \, dx\\
	+ \kappa \int_\Omega e^{- \|\Delta w_m\|_{L^2}^2 |\nabla v_m|^2} |\nabla v_m|^2 \, dx
	&= 1 + \omega^2 \int_\Omega |\nabla v_m|^2 \, dx +\kappa \, F(m)
	> 0
	\end{align*}
	for all sufficiently large $m \in \mathbb{N}$. 
	A contradiction. 
	Thus, we have shown $(A_1)$.	
	
	$(A_2)$ 
	\textit{There exists $\varphi \in X \setminus \{0\}$ such that $\|\Delta \varphi\|_{L^2} \geq \rho$ and $E_\omega(\varphi) < \alpha$, where $\rho$ and $\alpha$ as in $(A_1)$.}
	
	Recalling that $\kappa-\omega^2 > \lambda_1$ and considering $t \varphi_1$, where $\varphi_1$ is the first eigenfunction associated to $\lambda_1$, we get $E_\omega(t_1 \varphi_1) < 0$ for a sufficiently large $t_1 > 0$, see Lemma \ref{lem:ExNeh}.
	
	\smallskip	
	Hence, $E_\omega$ has the mountain-pass geometry. 
	Moreover, Lemma \ref{lem:PS} implies that $E_\omega$ satisfies the Palais--Smale condition.
	Consequently, applying \cite[Theorem 6.1]{struwe}, we see that $E_\omega$ possesses a mountain-pass-type critical point $u$ such that 
	\begin{equation}\label{eq:critlevel}
	E_\omega(u) = \beta = \inf_{g \in \Gamma} \max_{s \in [0,1]} E_\omega(g(s)) \geq \alpha > 0,
	\end{equation}
	where $\alpha$ is given in $(A_1)$ and 
	$$
	\Gamma := \{g \in C([0,1],X):~ g(0) = 0,~ E_\omega(g(1)) < 0\}.
	$$
	
	Let us show that $u$ is a ground state solution of \eqref{eq:pro1}. 
	Indeed, suppose, by contradiction, that there exists a solution $v$ such that $E_\omega(v) < E_\omega(u) = \beta$. According to Lemma \ref{lem:ExNeh}, we can find $t_2 > 0$ such that $E_\omega(t_2 v) < 0$. Taking $g(s) = s \, t_2 v$, we see that $g \in \Gamma$ and $\max\limits_{s \in [0,1]} E_\omega(g(s)) = E_\omega(v) < \beta$, which contradicts the definition of $\beta$.
\end{proof}

\begin{remark}
	The ground state solution $u$ of \eqref{eq:pro1} obtained in Proposition \ref{lem:1} is a solution of the minimization problem 
	\begin{equation}\label{N11}
	E_\omega(u) = \min\{E_\omega(v):~ v \in \mathcal{N}_\omega\}.
	\end{equation}
	Indeed, if we suppose, by contradiction, that there exists $v \in \mathcal{N}_\omega$ such that $E_\omega(v) < E_\omega(u) = \beta$, then, as at the end of the proof of Proposition \ref{lem:1}, we can consider $g(s) = s \, t_2 v \in \Gamma$, where $t_2 > 0$ is such that $E_\omega(t_2 v) < 0$. Therefore, by Lemma \ref{lem:ExNeh}, $\max\limits_{s \in [0,1]} E_\omega(g(s)) = E_\omega(v) < \beta$, which contradicts the definition of $\beta$.
\end{remark}

\begin{remark}\label{rem:crit1}
	Actually, it follows from Proposition \ref{lem:1} and Lemma \ref{lem:level} that any nonzero critical point $u$ of $E_\omega$ satisfies $E_\omega(u) \in \left[\beta, \frac{\kappa |\Omega|}{2}\right)$, where $\beta$ is the ground state level given by \eqref{eq:critlevel}.
\end{remark}

\section{Multiplicity of solutions}\label{sec:bound_state}

The multiplicity result of Theorem \ref{thm:2} follows from the following proposition.
\begin{prop}\label{prop:second_sol}
	Assume that $\lambda_{k} < \kappa - \omega^2 < \lambda_{k+1}$ for some $k \in \mathbb{N}$. Then $E_\omega$ has at least $k$ nonzero saddle-type critical points with positive energy.
\end{prop}
\begin{proof}
	To prove the result we will apply the abstract critical point theorem \cite[Theorem 2.4]{BBF} which is an improvement of 
	\cite[Theorem 2.23]{AR}. 
	Recall that $E_\omega$ is even, $E_\omega(0)=0$, and $E_\omega$ satisfies the Palais--Smale condition in $(0, +\infty)$ by Lemma \ref{lem:PS}, see \cite[p.\ 984]{BBF}.
	Therefore, to apply \cite[Theorem 2.4]{BBF} it suffices to show the existence of	 two closed subspaces $W$ and $V$ of $X$ with $\text{codim}\, V < +\infty$, and three constants $\rho > 0$, $\gamma > \alpha > 0$ such that 
	\begin{enumerate}
		\item[(i)] $E_\omega(u) \geq \alpha$ for all $u \in V$ with $\|\Delta u\|_{L^2} = \rho$;
		\item[(ii)] $E_\omega(u) < \gamma$ for all $u \in W$.
	\end{enumerate}
	Let $V := X$. Evidently, $\text{codim}\, V = 0$.
	Moreover, (i) is identical to $(A_1)$ of Proposition \ref{lem:1} and hence (i) is satisfied.
		
	In order to specify $W$, consider the basis $\{\varphi_1, \varphi_2, \dots\}$ for $X$ of orthonormal eigenfunctions of the problem \eqref{eq:eigenvalue}, see Section \ref{sec:preliminaries}. 
	Let us set $W := \text{span}\{\varphi_1, \dots, \varphi_k\}$, where $k \in \mathbb{N}$ is given by the assumption, and verify (ii).	
	Note that $\text{dim}\,W = k$, any $u \in W$ has the form $u = \sum_{i=1}^k \alpha_i \varphi_i$, and in view of the orthogonality of eigenfunctions
	\begin{align*}
	\int_\Omega |\Delta u|^2 \, dx = 
	\sum_{i=1}^k \alpha_i^2 \int_\Omega |\Delta \varphi_i|^2 \,dx 
	=
	\sum_{i=1}^k \alpha_i^2 \lambda_i \int_\Omega |\nabla \varphi_i|^2 \,dx 
	\leq \lambda_k \int_\Omega |\nabla u|^2 \, dx.
	\end{align*}
	Therefore, since $\lambda_k < \kappa-\omega^2$, there exists $t_0>0$ such that for any $u \in W$ with $\|\Delta u\|_{L^2} \geq t_0$ we have
	\begin{align*}
	\notag
	E_\omega(u) 
	&\leq \frac{\|\Delta u\|_{L^2}^2}{2}\left(1-\frac{\kappa-\omega^2}{\lambda_k}\right) - \frac{\kappa}{2} \int_\Omega e^{-|\nabla u|^2} \, dx + \frac{\kappa |\Omega|}{2} \\
	\label{eq:I<0}
	&\leq 
	\frac{t_0^2}{2}\left(1-\frac{\kappa-\omega^2}{\lambda_k}\right) 
	+ 
	\frac{\kappa |\Omega|}{2}
	\leq
	0.
	\end{align*}
	At the same time, it is evident that there exists sufficiently large $\gamma > 0$ such that $E_\omega(u) < \gamma$ for any $u \in W$ with $\|\Delta u\|_{L^2} \leq t_0$. The combination of the last two facts gives (ii).
	
	Thus, we apply \cite[Theorem 2.4]{BBF} to deduce that $E_\omega$ possesses at least $k$ distinct critical points of saddle type whose critical levels belong to $[\alpha,\gamma]$.
\end{proof}

\section{Regularity}\label{sec:regularity}
 
In this section, we will show $C^{4,\gamma}(\overline{\Omega})$-regularity of solutions obtained in Propositions \ref{lem:1} and \ref{prop:second_sol}. 
Recall that we assume $\partial \Omega \in C^{4,\eta}$.
In the case $X = W_0^{2,2}(\Omega)$, this regularity follows from \cite[Theorem 1]{luckhaus}. Therefore, we will prove this result for $X = W^{2,2}(\Omega) \cap W_0^{1,2}(\Omega)$.

\begin{prop}
	Let $X = W^{2,2}(\Omega) \cap W_0^{1,2}(\Omega)$.
	Then any solution of \eqref{eq:pro1} belongs to $C^{4,\gamma}(\overline{\Omega})$ for some $\gamma \in (0,1)$.
\end{prop}
\begin{proof}
	Let $u \in X$ be a weak solution of \eqref{eq:pro1}. 
	We rewrite \eqref{eq:pro1} as
	\begin{equation}\label{eq:P1}
	\Delta^2 u = f(x),
	\quad x \in \Omega,
	\end{equation}
	where
	\begin{equation}\label{eq:f}
	f(x) := \kappa \,\text{div} \left(e^{- |\nabla u(x)|^2} \nabla u(x) \right) - (\kappa-\omega^2) \Delta u(x).
	\end{equation}
	Note that $f \in L^2(\Omega)$. 
	Indeed, we have $\Delta u \in L^2(\Omega)$, and for the first term in \eqref{eq:f} we can write 
	\begin{align*}
	\text{div}\left(e^{-|\nabla u|^2} \nabla u\right) 
	= 
	e^{-|\nabla u|^2} 
	\left(
	\Delta u - 2 \sum_{i,j=1}^{N} \frac{\partial u}{\partial x_i}
	\frac{\partial u}{\partial x_j}
	\frac{\partial^2 u}{\partial x_i \partial x_j} 
	\right).
	\end{align*}
	Therefore, using the estimates $e^{-s^2}(1+2s^2) \leq 2$ and $|\partial u/\partial x_i| \leq |\nabla u|$ for $i=1,\dots,N$, we get
	\begin{equation}\label{eq:estimate_for_gaussian_diffusion}
	\left|\text{div}\left(e^{-|\nabla u|^2} \nabla u\right)\right| 
	\leq 
	e^{-|\nabla u|^2} (1 + 2 |\nabla u|^2) \sum_{|\alpha|=2} |D^\alpha u| \leq 2 \sum_{|\alpha|=2} |D^\alpha u|,
	\end{equation}
	which shows that $\text{div}\left(e^{-|\nabla u|^2} \nabla u\right) \in L^2(\Omega)$.
	
	\smallskip	
	Now we perform the following bootstrap argument. 
	
	\textit{Step 1}. Let $f \in L^{s_i}(\Omega)$ for some $s_i \geq 2$, $i=0,1,2,\dots$, and $s_0=2$. Then we can apply \cite[Theorem 2.20, p.\ 46]{gazzola2010} to \eqref{eq:P1} by taking $p=s_i$, $k=4$, $m=2$, $m_1 = 0$, $m_2 = 2$. This theorem implies the existence of a unique strong solution $w \in W^{4,s_i}(\Omega)$ for \eqref{eq:P1} subject to Navier boundary conditions. 
	Since the homogeneous problem 
	$$
	\Delta^2 v = 0,
	\quad x \in \Omega,
	$$
	with Navier boundary conditions admits only a trivial weak solution ($v=0$), we deduce that $w$ is a unique weak solution of \eqref{eq:P1}, that is, $w \equiv u$ a.e.\ in $\Omega$ and hence $u \in W^{4,s_i}(\Omega)$.
	
	\textit{Step 2}. If $N < 2s_i$, then go to Step 3. 
	If $N > 2s_i$, then we have $W^{4,s_i}(\Omega) \hookrightarrow W^{2,s_{i+1}}(\Omega)$, where $s_{i+1} := \frac{Ns_i}{N - 2s_i}$. 
	This implies that $f \in L^{s_{i+1}}(\Omega)$ as it follows from \eqref{eq:estimate_for_gaussian_diffusion}.
	Set $i := i+1$ and go to Step 1.
	
	Notice that the case $N = 2s_i$ can be handled analogously by taking an arbitrary $s_{i+1} \in (s_i,+\infty)$. We omit the details. 
	
	\textit{Step 3}. 
	Recall that $W^{4,s_i}(\Omega) \hookrightarrow C^{2,\tau}(\overline{\Omega})$ for some $\tau \in (0,1)$ whenever $N < 2s_i$. 
	Therefore, we see that $u \in C^{2,\tau}(\overline{\Omega})$. 
	Consequently, $f \in C^{0,\mu}(\overline{\Omega})$ for some $\mu \in (0,1)$, and hence \cite[Theorem 2.19, p.\ 45]{gazzola2010} implies that $u \in C^{4,\gamma}(\overline{\Omega})$ for some $\gamma \in (0,1)$.
	
	Let us show that Step 2 will lead to Step 3 after a finite number of iterations.
	Since $s_0=2$, we have $s_1 = \frac{2N}{N-4}$, and we can find $k \in \mathbb{N}$ such that $s_1  > \frac{N}{2k}$. Then, it is not hard to see that
	$$
	s_2 = \frac{N s_1}{N-2s_1} > \frac{N}{2(k-1)}, 
	\quad
	\dots,
	\quad  
	s_{i+1} = \frac{N s_i}{N-2s_i} > \frac{N}{2(k-i)},
	\quad
	\dots,
	\quad
	s_{k} > \frac{N}{2},
	$$
	and hence Step 2 leads to Step 3 after $k$ iterations.
\end{proof}

\section{Comparison with approximative problems}\label{sec:comparison}

In this section, we illustrate that existence results for the approximative problems \eqref{eq:prototype2} and \eqref{eq:prototype1} are substantially different from the existence results for the original problem \eqref{eq:pro1}.

\subsection{First approximative problem}

We consider the equation \eqref{eq:prototype2}:
\begin{equation*} 
\Delta^2 u -\omega^2 \Delta u + \kappa \, \text{div} \left(|\nabla u|^2 \nabla u\right) 
= 0
\end{equation*}
subject to Dirichlet or Navier boundary conditions \eqref{eq:Dirichlet} or \eqref{eq:Navier}, respectively.
The energy functional associated with \eqref{eq:prototype2} is defined by
$$
\mathcal{E}_1(u) 
= 
\frac{1}{2} \int_\Omega |\Delta u|^2 \, dx
+\frac{\omega^2}{2} \int_\Omega |\nabla u|^2 \, dx
- \frac{\kappa}{4} \int_\Omega |\nabla u|^4 \, dx.
$$
Even if the last integral is well-defined in the space $X$ for any $N \leq 4$, we treat the case $N \leq 3$ only.
\begin{prop}\label{prop:approx_first}
	Let $N \leq 3$. Then for any $\omega \in \mathbb{R}$ and $\kappa>0$, \eqref{eq:prototype2} has a ground state solution $u$ satisfying $\mathcal{E}_1(u) > 0$, which is a critical point of $\mathcal{E}_1$ of saddle type.
\end{prop}

To prove Proposition \ref{prop:approx_first}, consider the functional
$$
Q(u) := \left<\mathcal{E}_1'(u),u\right> =
\left.\frac{\partial}{\partial t} \mathcal{E}_1(t u)\right|_{t=1} = 
\int_\Omega |\Delta u|^2 \, dx 
+
\omega^2 \int_\Omega |\nabla u|^2 \, dx
- \kappa\int_\Omega |\nabla u|^4 \, dx,
$$
and the Nehari manifold corresponding to \eqref{eq:prototype2}:
$$
\mathcal{N} := 
\left\{ 
u \in X \setminus \{0\}:~ 
Q(u)=0
\right\}.
$$

First we need the following lemma.
\begin{lemma}\label{lem:Neh}
	For any $u \in X \setminus \{0\}$ there exists a unique $t_u > 0$ such that $t_u u \in \mathcal{N}$. Moreover, $t_u$ is a point of global maximum of $\mathcal{E}_1(tu)$ with respect to $t>0$, and $\mathcal{E}_1(t_u u)>0$.
	Furthermore, if $N \leq 3$, then for any $c>0$ there exists $u_c \in \mathcal{N}$ such that $\mathcal{E}_1(u_c) > c$.
\end{lemma}
\begin{proof}
	Let us fix any $u \in X \setminus \{0\}$ and consider the fibering functional
	\begin{equation}\label{eq:E1fib}
	\mathcal{E}_1(tu) 
	= \frac{t^2}{2} \left(\int_\Omega |\Delta u|^2 \, dx 
	+ \omega^2\int_\Omega |\nabla u|^2 \, dx \right) 
	- \frac{\kappa t^4}{4} \int_\Omega |\nabla u|^4 \, dx.
	\end{equation}
	It is not hard to see that there exists a unique $t_u > 0$ given by
	\begin{equation}\label{eq:tu}
	t_u := \left(\frac{\int_\Omega |\Delta u|^2 \, dx 
		+ \omega^2\int_\Omega |\nabla u|^2 \, dx}{\kappa  \int_\Omega |\nabla u|^4 \, dx}\right)^{1/2}
	\end{equation}
	such that $\frac{\partial}{\partial t} \mathcal{E}_1(t u) = 0$ at $t=t_u$, which implies that $t_u u \in \mathcal{N}$. 
	Further, analyzing directly the fibering functional, we derive that $t_u$ is a point of global maximum of $\mathcal{E}_1(tu)$ and $\mathcal{E}_1(t_u u)>0$. 
	
	To prove the last assertion we assume, without loss of generality, that $0 \in \Omega$, and take any $u \in C_0^\infty(\mathbb{R}^N)$ such that $\text{supp}\, u \subset B \subset \Omega$, where $B$ is a sufficiently small ball centered at the origin. Finding $t_u$ by \eqref{eq:tu} and substituting it to \eqref{eq:E1fib}, we get
	$$
	\mathcal{E}_1(t_u u) = \frac{\left(\int_\Omega |\Delta u|^2 \, dx 
		+ \omega^2\int_\Omega |\nabla u|^2 \, dx \right)^2}{2 \kappa \int_\Omega |\nabla u|^4 \, dx}.
	$$
	Consider now the function $u_\sigma$ defined by $u_\sigma(x) := u(x/\sigma)$ for $\sigma \in (0,1)$. We see that $u_\sigma \in C_0^\infty(B)$. Moreover, we obtain
	$$
	\mathcal{E}_1(t_{u_\sigma} u_\sigma) = \frac{\left(\sigma^{N-4}\int_B |\Delta u|^2 \, dx 
		+ \omega^2 \sigma^{N-2} \int_B |\nabla u|^2 \, dx \right)^2}{2 \kappa \sigma^{N-4} \int_B |\nabla u|^4 \, dx}
	\geq
	\sigma^{N-4}\frac{\left(\int_B |\Delta u|^2 \, dx \right)^2}{2 \kappa \int_B |\nabla u|^4 \, dx}.
	$$
	Therefore, recalling that $N\leq 3$, we deduce that for any $c>0$ there exists sufficiently small $\sigma > 0$ such that 
	$t_{u_\sigma} u_\sigma \in \mathcal{N}$ satisfies $\mathcal{E}_1(t_{u_\sigma} u_\sigma) > c$.
\end{proof}

\begin{cor}
	Any nonzero critical point of $\mathcal{E}_1$ has positive energy.
\end{cor}

\smallskip
\noindent \textit{Proof of Proposition \ref{prop:approx_first}}.
	Consider the minimization problem
	\begin{equation*}\label{N22}
	\hat{\mathcal{E}}_1 := 
	\min
	\left\{
	\mathcal{E}_1(u): ~ u \in  \mathcal{N}
	\right\}.
	\end{equation*}
	To prove the proposition it is sufficient to show that for any $\omega \in \mathbb{R}$ and $\kappa>0$, $\hat{\mathcal{E}}_1$ is achieved and any corresponding minimizer is a ground state solution of \eqref{eq:prototype2}.

	From Lemma~\ref{lem:Neh} it readily follows that $\mathcal{N} \neq \emptyset$. Thus, there exists a minimizing sequence $\{u_n\}_{n \in \mathbb{N}} \subset \mathcal{N}$ for $\hat{\mathcal{E}}_1$. First we show that $\{u_n\}_{n \in \mathbb{N}}$ is bounded in $X$.	
	Suppose, by contradiction, that $\|\Delta u_n\|_{L^2} \to +\infty$ as $n \to +\infty$. 
	Then, by the Nehari constraint $Q(u_n)=0$, we get
	$$
	\kappa \|\nabla u_n\|_{L^4}^4 = \|\Delta u_n\|_{L^2}^2 + \omega^2\|\nabla u_n\|_{L^2}^2 \to +\infty
	\quad \text{as } n \to +\infty.
	$$
	This implies that
	$$
	\mathcal{E}_1(u_n) = \frac{\kappa}{4} \int_\Omega |\nabla u_n|^4 \, dx \to +\infty
	\quad \text{as } n \to +\infty,
	$$
	which contradicts a minimizing property of $\{u_n\}_{n \in \mathbb{N}}$. 
	Thus, $\{u_n\}_{n \in \mathbb{N}}$ is bounded in $X$ and hence $u_n$ converges, up to a subsequence, to some $u \in X$ weakly in $X$ and strongly in $W^{1,2}(\Omega)$ and $W^{1,4}(\Omega)$, since $N \leq 3$, cf. Section \ref{sec:preliminaries}.
	
	The embedding $W^{2,2}(\Omega) \hookrightarrow W^{1,4}(\Omega)$ implies the existence of $c_1 > 0$ such that
	$$
	c_1 \|\nabla u_n\|_{L^4} \leq \|\Delta u_n\|_{L^2}
	\quad
	\text{for all }
	n \in \mathbb{N}.
	$$
	Then, by the Nehari constraint $Q(u_n)=0$, we have
	$$
	c_1 \|\nabla u_n\|_{L^4}^2 \leq \|\Delta u_n\|_{L^2}^2 = \kappa\|\nabla u_n\|_{L^4}^4 - \omega^2\|\nabla u_n\|_{L^2}^2 \leq \kappa \|\nabla u_n\|_{L^4}^4
	\quad
	\text{for all }
	n \in \mathbb{N}.
	$$
	Therefore, we conclude that $\|\nabla u_n\|_{L^4} > c_2 > 0$ for some $c_2 > 0$ and all $n \in \mathbb{N}$, that is, $u \not\equiv 0$ a.e.\ in $\Omega$, due to the strong convergence $u_n \to u$ in $W^{1,4}(\Omega)$.
	
	Let us show now that $u_n \to u$ strongly in $X$. Suppose, by contradiction, that
	$$
	\|\Delta u\|_{L^2} < \liminf_{n \to +\infty} \|\Delta u_n\|_{L^2}.
	$$
	Recalling that $u_n \to u$ strongly in $W^{1,2}(\Omega)$ and $W^{1,4}(\Omega)$, we get $Q(u) < 0$.
	However, from Lemma~\ref{lem:Neh} we know that there exists $t_u \in (0,1)$, such that $t_u u \in \mathcal{N}$. 
	Since $t_{u_n} = 1$ for each $n \in \mathbb{N}$ and $t_{u_n}$ is a point of global maximum of each $\mathcal{E}_1(t u_n)$ with respect to $t>0$, we have
	$$
	\hat{\mathcal{E}}_1 \leq \mathcal{E}_1(t_u u) < \liminf_{n \to +\infty} \mathcal{E}_1(t_u u_n) \leq \liminf_{n \to +\infty} \mathcal{E}_1(u_n) = \hat{\mathcal{E}}_1,
	$$
	which is impossible. Hence, $u_n \to u$ strongly in $X$ and $u \in \mathcal{N}$.
	
	Finally, let us show that $u$ is a critical point of $\mathcal{E}_1$. 
	Indeed, applying the Lagrange multiplier rule (see, e.g., \cite[Theorem 48.B]{zeidler}), we obtain $\nu_1, \nu_2 \in \mathbb{R}$ such that $|\nu_1|+|\nu_2|>0$ and 
	$$
	\nu_1 \left<\mathcal{E}_1'(u),\varphi\right>
	+
	\nu_2 \left<Q'(u),\varphi\right> = 0
	\quad
	\text{for all }
	\varphi \in X.
	$$
	Taking $\varphi = u$ and noting that
	\begin{align*}
	\left<Q'(u), u \right>
	= 
	2\int_\Omega |\Delta u|^2 \, dx + 2\omega^2\int_\Omega |\nabla u|^2 \, dx - 4\kappa \int_\Omega |\nabla u|^4 \, dx
	=-2\kappa\int_\Omega |\nabla u|^4 \, dx < 0,
	\end{align*}
	we conclude $\nu_2 = 0$, $\nu_1 \neq 0$, and hence $\left<\mathcal{E}_1'(u), \varphi\right>=0$ for all $\varphi \in X$, that is, $u$ is a critical point of $\mathcal{E}_1$.
\qed

\medskip
\begin{remark}
We see that, unlike \eqref{eq:pro1}, the problem \eqref{eq:prototype2} admits a nonzero solution without the assumption $\kappa-\omega^2 > \lambda_1$. 
\end{remark}

\subsection{Second approximative problem}

Consider now the problem \eqref{eq:prototype1}:
\begin{equation*}
\Delta^2 u -\omega^2 \Delta u + \kappa \, \text{div} \left(|\nabla u|^2 \nabla u\right) 
-\frac{\kappa}{2} \, \text{div} \left(|\nabla u|^4 \nabla u\right)
= 0.
\end{equation*}
subject to Dirichlet or Navier boundary conditions \eqref{eq:Dirichlet} or  \eqref{eq:Navier}, respectively.
The energy functional associated with \eqref{eq:prototype1},
$$
\mathcal{E}_2(u) 
= 
\frac{1}{2} \int_\Omega |\Delta u|^2 \, dx
+\frac{\omega^2}{2} \int_\Omega |\nabla u|^2 \, dx
- \frac{\kappa}{4} \int_\Omega |\nabla u|^4 \, dx
+ \frac{\kappa}{12} \int_\Omega |\nabla u|^6 \, dx
$$
is well-defined in $X$ for $N \leq 3$ and in $X \cap W^{1,6}(\Omega)$ for $N \geq 4$.
\begin{prop}\label{prop:approx_two}
	Let $N \geq 1$. Then for any $\omega \in \mathbb{R}$ there exists sufficiently large $\kappa > 0$ such that $\mathcal{E}_2$ possesses a nonzero global minimizer $u$ with $\mathcal{E}_2(u) < 0$.
\end{prop}
\begin{proof}
	Consider the first eigenfunction $\varphi_1$ of the problem \eqref{eq:eigenvalue}. 
	Since $\varphi \in C^4(\overline{\Omega})$, we have
	$$
	\mathcal{E}_2(t\varphi_1) 
	= 
	\frac{(\lambda_1+\omega^2)t^2}{2} \int_\Omega |\nabla \varphi_1|^2 \, dx
	- \frac{\kappa t^4}{4} \int_\Omega |\nabla \varphi_1|^4 \, dx
	+ \frac{\kappa t^6}{12} \int_\Omega |\nabla \varphi_1|^6 \, dx.
	$$
	For any $\omega \in \mathbb{R}$ there exist $\kappa > 0$ (large enough) and $t>0$ such that 
	\begin{equation}\label{eq:neg}
	\frac{\mathcal{E}_2(t\varphi_1)}{t^2}
	= 
	\frac{(\lambda_1+\omega^2)}{2} \int_\Omega |\nabla \varphi_1|^2 \, dx
	- \frac{\kappa t^2}{4}\left(\int_\Omega |\nabla \varphi_1|^4 \, dx
	- \frac{t^2}{3} \int_\Omega |\nabla \varphi_1|^6 \, dx\right) < 0.
	\end{equation}
	Moreover, it is not hard to show that $\mathcal{E}_2$ is weakly lower-semicontinuous and coercive in $X$ for $N \leq 3$ and in 
	$X \cap W^{1,6}(\Omega)$ for $N \geq 4$. Therefore, the direct method of calculus of variations implies the existence of a global minimizer $u$ of $\mathcal{E}_2$ and \eqref{eq:neg} entails $\mathcal{E}_2(u) < 0$. 
\end{proof}

\begin{remark}
The result of Proposition \ref{prop:approx_two} states that a ground state solution $u$ of the problem \eqref{eq:prototype1} for sufficiently large $\kappa>0$ satisfies $\mathcal{E}_2(u) < 0$ and it is not a critical point of saddle type, in contrast with problems \eqref{eq:pro1} and \eqref{eq:prototype2}.
\end{remark}

\section{Conclusions}\label{sec:conclusion}

We proved the existence, nonexistence, and multiplicity results for the stationary Zakharov equation \eqref{eq:pro1} in a bounded domain. 
We also discussed approximative problems \eqref{eq:prototype2} and \eqref{eq:prototype1} considered previously in the literature and found striking differences among all of these problems. 

In particular, our nonexistence result states that there are no nonzero solutions of \eqref{eq:pro1} for $\omega^2$ large enough. Analysis of the Nehari manifold similar to that from Section \ref{sec:preliminaries} indicates a nonexistence result for large $\omega^2$ when we consider nonzero solutions of \eqref{eq:pro1} in the entire space $\mathbb{R}^N$ decaying at infinity.
Let us emphasize that corresponding physical models should possess similar nonexistence property as it was observed, e.g., in \cite[p.\ 86]{dyachenko1991}.
However, a model described by \eqref{eq:prototype2} does not have this property as shown in Proposition \ref{prop:approx_first}. See \cite{colin} for the entire space case $\mathbb{R}^N$.

Another difference is that ground state solutions of \eqref{eq:pro1} and \eqref{eq:prototype2} are saddle-type critical points of associated energy functionals, but the ground state solution of the second approximative problem \eqref{eq:prototype1} can be a global energy minimizer. Moreover, critical energy levels for the original problem \eqref{eq:pro1} and approximative problems \eqref{eq:prototype2} and \eqref{eq:prototype1} exhibit different properties. 
Namely, the critical levels associated with \eqref{eq:pro1} are bounded, see Remark \ref{rem:crit1}, whereas there is no such a restriction for the problem \eqref{eq:prototype2}, see Lemma \ref{lem:Neh}. Furthermore, ground state levels of \eqref{eq:pro1} and \eqref{eq:prototype2} are always positive, while the ground state level of \eqref{eq:prototype1} may be negative.

\smallskip
\bigskip
\noindent
\textbf{Acknowledgments.}
V. Bobkov and P. Dr\'abek were supported by the grant 18-03253S of the Grant Agency of the Czech Republic. V. Bobkov and Y. Ilyasov were also supported by the project LO1506 of the Czech Ministry of Education, Youth and Sports.
Y. Ilyasov wishes to thank the University of West Bohemia, where this research was started, for the invitation and hospitality.

\end{document}